\documentclass[a4paper,twoside,9pt]{article}
\usepackage{amsmath,amssymb,graphics,epsfig,color}
\usepackage{amsmath,amsthm,amsfonts,latexsym,amscd,amssymb,enumerate,wasysym,stmaryrd,skull}

\usepackage{graphicx}
\usepackage[all]{xy}
\usepackage{fancyhdr}
\usepackage[T1]{fontenc}
\usepackage{ifthen}
\usepackage{amsmath,amsfonts,amssymb}
\newcommand{\pleinepage}%
{\setlength{\oddsidemargin}{0in}\setlength{\textwidth}{6.26in}\setlength{\topmargin}{0in}\setlength{\textheight}{8.7in}}
\newcommand{\grossepage}%
{\setlength{\oddsidemargin}{-0.5cm}\setlength{\textwidth}{17.5cm}\setlength{\topmargin}{-1.5cm}\setlength{\textheight}{24cm}}
\newcommand{\defit}[1]%
{{\em #1}}

\def\Re{\mathop{\plainRe\mkern -2mu\mit e}\nolimits}
\def\Im{\mathop{\plainIm\mkern -2mu\mit m}\nolimits}
\def\surl#1_#2{\mathrel{\mathop{\kern 0pt #1}\limits_{#2}}}

\newcommand{\fleche}[1]%
{\rTo^{#1}}
\newcommand{\fonction}[5]%
{\begin{diagram}
#2 & {} &\rTo^{#1} & {} & #3 \\
#4 & {} &\rMapsto & {} & #5 
\end{diagram} }
\newcommand{\sfonction}[5]%
{$\begin{array}{ccc}#2 & {\buildrel #1 \over \rightarrow} & #3 \\#4 & \mapsto & #5 \\ \end{array}$ }
\newcommand{\accolade}[1]%
{\begin{cases}  #1 \end{cases}}

\newcounter{nbre}
\newcommand{\entete}[6]%
{{\large \noindent%
\mbox{\begin{tabular}{c} #1 \\ #2 \end{tabular}}\hspace{\fill}\mbox{\begin{tabular}{c} #3 \\ #4 \end{tabular}}\vspace{1cm}\begin{center}{\Huge \textsc{#5}} \\ \vspace{0.5cm}\begin{tabular}{c} #6 \\ \hline \end{tabular}\end{center}\bigskip}}
\newcommand{\defifont}{ \sc }

{\refstepcounter{compteur} \par \vspace{0.3cm}\ \\ \noindent {{  \textsc{\textbf{Definitions}}} 
    \textbf{\thecompteur} \ 
    }---\ \sffamily\renewcommand{\em}{\normalfont\itshape}}{\par
    \vspace{0.3 cm}}

\newcommand{\gloss}[1]%
 {\index{{#1}@{#1}}{\em #1}\relax}
\newcommand{\xgloss}[2]%
 {\index{{#1}@{#1}!{#2}@{#2}}{\em #1\relax #2}\relax}
 \newcommand{\glossref}[2]%
 {\index{{#2}@{#1}}{\em #1}\relax}
 \newcommand{\xglossref}[4]%
 {\index{{#3}@{#1}!{#4}@{#2}}{\em #1 \relax #2 }\relax}
\newcounter{compteur}
\renewcommand{\thecompteur}{\thesection.\arabic{compteur}}
\newenvironment{dfn}[1][]%
{\refstepcounter{compteur} \par \vspace{0.3cm}\ \\ \noindent {{  \textsc{\textbf{Definition}}} 
    \textbf{\thecompteur} \ 
    }---\ \sffamily\renewcommand{\em}{\normalfont\itshape}}{\par
    \vspace{0.3 cm}}
{\refstepcounter{compteur} \par \vspace{0.3cm}\ \\ \noindent {{  \textsc{\textbf{Definitions}}} 
    \thecompteur \ 
    }---\ \sffamily\renewcommand{\em}{\normalfont\itshape}\begin{enumerate}}{\end{enumerate}
    \par \vspace{0.5 cm}}
\newcounter{theonum}

\newenvironment{thm}[1][]%
{\refstepcounter{theonum} \par \vspace{0.3cm}\ \\ \noindent {{  \textsc{\textbf{Theorem}}} 
    \textbf{\thecompteur} \ 
    }---\ \sffamily\renewcommand{\em}{\normalfont\itshape}}{\par
    \vspace{0.3 cm}}

\newcommand\addpage[2]{#2, page #1}
\makeatletter
\renewcommand\p@theonum{\protect\addpage{\thepage}}
\makeatother
 
\newenvironment{prop}[1][]%
{\refstepcounter{compteur} \par \vspace{0.2cm}\ \\ \noindent {{  \textsc{\textbf{Proposition}}} 
    \textbf{\thecompteur} \ 
    }---\ \sffamily\renewcommand{\em}{\normalfont\itshape}}{\par
    \vspace{0.2 cm}}

\newenvironment{prop*}[1][]%
{ \par \vspace{0.2cm}\ \\ \noindent {{\textsc{\textbf{Proposition}}} 
    #1 \ 
    }---\sffamily\renewcommand{\em}{\normalfont\itshape}}{\par \vspace{0.3 cm}}
{\refstepcounter{compteur} \par \vspace{0.3cm}\ \\ \noindent {{  \textsc{\textbf{Properties}}} 
    \thecompteur \ 
    }---\ \sffamily\renewcommand{\em}{\normalfont\itshape}}{\par
    \vspace{0.3 cm}}

{\refstepcounter{compteur} \par \vspace{0.3cm}\ \\ \noindent {{ \textsc{\textbf{Corollary}}}
    \textbf{\thecompteur} 
    \ }---\ \sffamily\renewcommand{\em}{\normalfont\itshape}}{\par  \vspace{0.3 cm}}

{\refstepcounter{compteur} \par \vspace{0.2cm}\ \\ \noindent {{  \textsc{\textbf{Lemma}}} 
    \textbf{\thecompteur} \ 
    }---\ \sffamily\renewcommand{\em}{\normalfont\itshape}}{\par
    \vspace{0.2 cm}}    
    
\renewenvironment{proof}%
{\par \vspace{0.2cm}\ \\ \noindent{ { \textsc{Proof}}\,---\ } }{\hfill{$\Box$} \par \vspace{0.2 cm}}

{\par \vspace{0.2cm}\ \\ \noindent{\sc \textbf{Remark}\ }---\ }{\par \vspace{0.2cm}}
{\refstepcounter{compteur}\par \vspace{0.2cm}\ \\ \noindent{\sc \textbf{Conjecture} \textbf{\thecompteur}\ }---\ }{\par \vspace{0.2cm}}

\begin{document}

\NoCompileMatrices
% fibrations
\def\ds{\displaystyle}
\def\pn{\pi_{n}}
\def\pnu{\pi_{n-1}}
\pagestyle{fancy}
% with this we ensure that the chapter and section
% headings are in lowercase.
%\renewcommand{\chaptermark}[1]{\markboth{#1}{}}
\renewcommand{\sectionmark}[1]{\markright{\thesection\ #1}}
\fancyhf{} % delete current setting for header and footer
%\fancyhead[LE,RO]{\;\thepage}

%\fancyhead[CO]{\textsc{}}
%\fancyhead[CE]{\textsc{Paolo Antonini}}

\renewcommand{\headrulewidth}{0.16pt}
\renewcommand{\footrulewidth}{0pt}
\addtolength{\headheight}{0.7pt} % make space for the rule
%\fancypagestyle{plain}{%
%\fancyhead{} % get rid of headers on plain pages
\renewcommand{\headrulewidth}{0pt} % and the line

% quando il footnote e' della stessa dimensione:
\newcommand{\foot}[1]{\footnote{\begin{normalsize}#1\end{normalsize}}}

% o p e r a t o r i     m a t e m a t i c i    

\def\bX{\partial X}
\def\dim{\mathop{\rm dim}}
\def\Re{\mathop{\rm Re}}
\def\Im{\mathop{\rm Im}}
\def\I{\mathop{\rm I}}
\def\Id{\mathop{\rm Id}}
\def\grad{\mathop{\rm grad}}
\def\vol{\mathop{\rm vol}}
\def\SU{\mathop{\rm SU}}
\def\SO{\mathop{\rm SO}}
\def\Aut{\mathop{\rm Aut}}
\def\End{\mathop{\rm End}}
\def\GL{\mathop{\rm GL}}
\def\Cinf{\mathop{\mathcal C^{\infty}}}
\def\Ker{\mathop{\rm Ker}}
\def\Coker{\mathop{\rm Coker}}
\def\dom{\mathop{\rm Dom}}
\def\Hom{\mathop{\rm Hom}}
\def\Ch{\mathop{\rm Ch}}
\def\sign{\mathop{\rm sign}}
\def\SF{\mathop{\rm SF}}

\def\AS{\mathop{\rm AS}}
\def\spec{\mathop{\rm spec}}
\def\Ric{\mathop{\rm Ric}}
\def\ch{\mathop{\rm ch}}
\def\Ch{\mathop{\rm Ch}}

\def\ev{\mathop{\rm ev}}
\def\id\textrm{Id}
\def\dd{\mathcal{D}(d)}
\def\Cli{\mathbb{C}l(1)}
\def\kerd{\operatorname{ker}(d)}

%                  l e t t e r e     g r e c h e 
\def\Fi{\Phi}

\def\de{\delta}
\def \dl{\partial L_x^0}
\def\e{\eta}
\def\ep{\epsilon}
\def\ro{\rho}
\def\a{\alpha}
\def\o{\omega}
\def\O{\Omega}
\def\b{\beta}
\def\la{\lambda}
\def\th{\theta}
\def\s{\sigma}
\def\t{\tau}
\def\g{\gamma}
\def\D{\Delta}
\def\G{\Gamma}
\def \fol{\mathcal F}
\def\R{\mathbin{\mathbb R}}
\def\Rn{\R^{n}}
\def\C{\mathbb{C}}
\def\Cm{\mathbb{C}^{m}}
\def\Cn{\mathbb{C}^{n}}
\def\gr{\mathcal{G}}
% Some special symbols
\def\Kahler{{K\"ahler}}
\def\w{{\mathchoice{\,{\scriptstyle\wedge}\,}{{\scriptstyle\wedge}}
{{\scriptscriptstyle\wedge}}{{\scriptscriptstyle\wedge}}}}
% Calligraphic and bold abbreviations
\def\cA{{\cal A}}\def\cL{{\cal L}}
\def\cO{{\cal O}}\def\cT{{\cal T}}\def\cU{{\cal U}}
\def\cD{{\cal D}}\def\cF{{\cal F}}\def\cP{{\cal P}}\def\cH{{\cal H}}\def\cL{{\cal L}}
\def\cB{{\cal B}}

% N O R M A   D I    U N    V E T T O R E 

\newcommand{\n}[1]{\left\| #1\right\|}% grande palla B di raggio R

%%%%%%%%%%%%%%%%%%%%%%%%%%%%%%%%%%%%%%%%%%%%%%%%%%%%%%%%%%%%%%%%%%%%%%%%%%%%%%%%
%%%%%%%%%%%%%%%%%%%%%%%%%%%%%%%%%%%%%%%%%%%%%%%%%%%%%%%%%%%%%%%%%%%%%%%%%%%%%%%5
\def\Z{\mathbb{Z}}
\def\cgs{C^{*}(\Gamma,\sigma)}
\def\bcgs{C^{*}(\Gamma,\bar{\sigma})}
\def\cgsr{C^{*}_{red}(,\sigma)}
\def\Mt{\tilde{M}}
\def\Et{\tilde{E}}
\def\Vt{\tilde{V}}
\def\Xt{\tilde{X}}
\def\N{\mathbb{N}}
\def\Nbs{\N^{\bar{\s}}}
\def\rcab{\ro^{[c]}_{\a-\b}}
\def\rc{\ro^{[c]}}
\def\Cd{\mathbb{C}^{d}}
\def\tr{\mathop{\rm tr}}\def\tralg{\tr{}^{\text{alg}}}       

%%%%%%        T R A C C E   %%%%%%%%%%%%%%%
\def\TR{\mathop{\rm TR}}\def\trace{\mathop{\rm trace}}
\def\STR{\mathop{\rm STR}}
\def\trG{\mathop{\rm tr_\Gamma}}
\def\TRG{\mathop{\rm TR_\Gamma}}
\def\Tr{\mathop{\rm Tr}}
\def\Str{\mathop{\rm Str}}
\def\Cl{\mathop{\rm Cl}}
\def\Op{\mathop{\rm Op}}
\def\supp{\mathop{\rm supp}}
\def\scal{\mathop{\rm scal}}
\def\ind{\mathop{\rm ind}}
\def\Ind{\mathop{\mathcal I\rm nd}\,}
\def\Diff{\mathop{\rm Diff}}
\def\T{\mathcal{T}}
\def\dn{\textrm{dim}_{\Lambda}}
\def \lke{\textrm L^2-\textrm{Ker}}

%Connessione
\newcommand{\bcot}{{}^bT^*X}
\newcommand{\wt}{\widetilde}
\newcommand{\go}{\mathcal{G}^{0}}
\newcommand{\dii}{(d_x^{k-1})^\ast}
\newcommand{\di}{d_x^{k-1}}
\newcommand{\ra}{\operatorname{range}}
\newcommand{\rb}{\rangle}
\newcommand{\lb}{\langle}
\newcommand{\re}{\mathcal{R}}
\newcommand{\vo}{\operatorname{End}_{\Lambda}(E)}
\newcommand{\mt}{\mu_{\Lambda,T}}
\newcommand{\tru}{\operatorname{tr}_{\Lambda}}
\newcommand{\buno}{B^1_{\Lambda}(E)}
\newcommand{\bdue}{B^2_{\Lambda}(E)}
\newcommand{\clis}{H^{2k}_{(2),dR}(L_x^0)}
\newcommand{\cali}{L^2(\Omega^{2k}(\partial L_x^0))}
\newcommand{\binf}{B^{\infty}_{\Lambda}(E)}
\newcommand{\bif}{B^{f}_{\Lambda}(E)}
\newcommand{\bsupp}{\operatorname{WF}'_b}
\newcommand{\vn}{\operatorname{End}_{\mathcal{R}}}
\newcommand{\ho}{\operatorname{Hom}_{\Lambda}}
\newcommand{\spc}{\operatorname{spec}_{\Lambda,e}}
\newcommand{\ix}{\operatorname{Ind}_{\Lambda}}
\newcommand{\cic}{C^{\infty}_c(L_x;E_{|L_x})}
\newcommand{\ci}{C^{\infty}_c(L_x;E_{|L_x})}
\newcommand{\tx}{\{T_x\}_{x\in X}}
\newcommand{\cc}{C^{\infty}_c(X)}
\newcommand{\rom}{\underline{\mathcal{R}_0}}
\newcommand{\roma}{(\mathcal{R}_0)_{|\partial X_0}     }
\newcommand{\dfo}{D^{\mathcal{F}_{\partial}}}
\newcommand{\deu}{D_{\epsilon,u}}
\newcommand{\deupp}{D^{+}_{\epsilon,u}}
\newcommand{\deum}{D^{-}_{\epsilon,u}}
\newcommand{\deuf}{D_{\epsilon,u}^{\mathcal{F}_{\partial}}}
\newcommand{\deufo}{D_{\epsilon,u,x_0}^{\mathcal{F}_{\partial}}}
\newcommand{\pie}{\Pi_{\epsilon}}
\newcommand{\pal}{\partial L_x}
\newcommand{\pr}{\partial_r}
\newcommand{\inbl}{\int_{\partial L_x} }
\newcommand{\pkp}{\chi_{\{0\}}(D^+_x)}
\newcommand{\deup}{D_{\epsilon,x}^{\pm}}
\newcommand{\dext}{D_{\epsilon,\mp u,x}^{\pm}}
\newcommand{\dex}{D_{\epsilon,\pm u,x}^{\pm}}
\newcommand{\ext}{\operatorname{Ext}(D_{\epsilon,x}^{\pm})}
\newcommand{\eppu}{0<|u|<\epsilon}
\newcommand{\hdeupx}{e^{-tD^2_{\epsilon,u,x}}}
\newcommand{\udif}{\operatorname{UDiff}}
\newcommand{\ki}{L^2(\Omega^kL_x^0)}
\newcommand{\uc}{\operatorname{UC}}
\newcommand{\op}{\operatorname{Op}}
\newcommand{\deux}{D_{\epsilon,u,x}}
\newcommand{\pk}{\phi_k}
\newcommand{\hdeupsx}{e^{-sD^2_{\epsilon,u,x}}}
\newcommand{\hdeups}{e^{-sD^2_{\epsilon,u}}}
\newcommand{\hdeut}{e^{-tD^2_{\epsilon,u}}}
\newcommand{\indu}{\operatorname{ind}_{\Lambda}}
\newcommand{\stru}{\operatorname{str}_{\Lambda}}
\newcommand{\deuq}{D^2_{\epsilon,u}}
\newcommand{\intk}{\int_{\sqrt{k}}^{\infty}}
\newcommand{\dmd}{d\mu_{\Lambda,D_{\epsilon,u}}(x)}
\newcommand{\defox}{D_{x}^{\mathcal{F}_{\partial}}}
\newcommand{\mun}{\mu_{\Lambda,D_{\epsilon,u}}(x)}
\newcommand{\tsi}{\int_{-\sigma}^{\sigma}}
\newcommand{\ak}{\lim_{k\rightarrow \infty}\operatorname{LIM}_{s\rightarrow 0}}
\newcommand{\deus}{D_{\epsilon,u}e^{-tD_{\epsilon,u}^2}}

\newcommand{\deuss}{D_{\epsilon,u}^2e^{-tD_{\epsilon,u}^2}}
\newcommand{\pkd}{\phi_k^2}
\newcommand{\eup}{e^{-tD^{+}_{\epsilon,u}D^{-}_{\epsilon,u}}}
\newcommand{\eum}{e^{-tD^{-}_{\epsilon,u} D^{+}_{\epsilon,u} }}
\newcommand{\deussx}{D_{\epsilon,u,x}e^{-tD_{\epsilon,u,x}^2}}
\newcommand{\dessx}{S_{\epsilon,u,x}e^{-tS_{\epsilon,u,x}^2}}
\newcommand{\clib}{c(\partial_r)\partial_r \phi_k^2}
\newcommand{\sk}{\int_s^{\sqrt{k}}}
\newcommand{\esm}{S_{\epsilon,u}e^{-tS_{\epsilon,u}^2}}
\newcommand{\deussxo}{D_{\epsilon,u,z_0}e^{-tD_{\epsilon,u,x_0}^2}}
\newcommand{\dessxo}{S_{\epsilon,u,z_0}e^{-tS_{\epsilon,u,z_0}^2}}
\newcommand{\deusszo}{D^{\mathcal{F}_{\partial}}_{\epsilon,u,z_0}e^{-t(D^{\mathcal{F}_{\partial}}_{\epsilon,u,x_0})2}}
\newcommand{\desszo}{S_{\epsilon,u,x_0}e^{-tS_{\epsilon,u,x_0}^2}}
\newcommand{\essp}{S_{\epsilon,u}^2}
\newcommand{\esspo}{S_{0,u}^2}
\newcommand{\dotto}{\dot{\theta}}
\newcommand{\piep}{\Pi_{\epsilon}}
\newcommand{\ome}{\Omega}
\newcommand{\deffo}{D^{\mathcal{F}_{\partial}}}
\newcommand{\nablal}{\nabla_x^l}
\newcommand{\nablak}{\nabla_y^k}
\newcommand{\kerk}{[f(P)]_{(x_0,\bullet)} }
\newcommand{\kepp}{\operatorname{Ker} (D^{\mathcal{F}_0^+})}
\newcommand{\bcm}{\Psi_{bc}^m(X)}
%PECETTA
\newcommand{\ty}{\infty}
\definecolor{light}{gray}{.95}
\newcommand{\pecetta}[1]{
$\phantom .$
\bigskip
\par\noindent
\colorbox{light}{\begin{minipage}{13.5 cm}#1\end{minipage}}
\bigskip
\par\noindent
}

%D I R A C :
\newcommand\Di{D\kern-7pt/}
%%%%%%%%%%%%%%%%%%%%%%%%%%%%%%%%%%%%%%%%%%%%%%%%%%%%%%%%%%%%%%%%%%%%%%%%%%%%%%
%%%%%%%%%%%%%%%%%%%%%%%%%%%%%%%%%%%
%         E n d    o f    T o p m a t t e r   %%%%%%%%%

\title{Generalized Dirac operators on Lorentzian manifolds and propagation of singularities}
\author{\Large Paolo Antonini\\
antonini@mat.uniroma1.it
\\
paolo.anton@gmail.com}

\maketitle
\begin{abstract}
We survey the correct definition of a generalized Dirac operator on a Space--Time and the classical result about propagation of singularities. This says that light travels along light--like geodesics. Finally we show this is also true for generalized Dirac operators.
\end{abstract}
\tableofcontents
\section{Introduction}
The celebrated theorem of Nils Denker about the propagation of singularities of a real principal type system when applied to the Dirac operator on a Lorentzian manifolds says the well--known fact that light travels along light--light geodesics with the light speed. In this paper first we give an appropiate definition of a generalized Dirac operator on a Lorentzian manifold then we show that the same result is true for generalized Dirac operator. This means that the Denker (partial) connection on the polarization set (along Hamiltonian orbits) of the system is the starting connection lifted to the cotangent bundle. We thank Bernd Ammann for giving us the correct definition of the generalized Dirac operator.

\section{The generalized Dirac operators on Lorentzian manifolds}
We start with a definition \cite{ginoux}
Let $(X,g)$ be a Lorentzian manifold.
With $\operatorname{Cl}(X)$ denote the bundle over $X$ whose fiber over $p$ is the Clifford algebra of $T_pX$ i.e the quotient of the complexified tensor algebra of $T_pX$ by the bilateral ideal generated by the elements of the form $v\otimes w+w\otimes v+2g(v,w)$. The Lorentzian metric and the connection extend to $\operatorname{Cl}(X)$ to give a metric connection which is Leibnitz with respect to the Clifford multiplication.
\begin{dfn}
A Clifford module of spin $k/2$ over $(X,g)$ is given by a complex vector bundle $S$ over $X$ with a sesquilinear product (antilinear in the second) $\langle \cdot, \cdot\rangle$ with connection $\nabla^S$ together with a smooth section $Q$ of $\operatorname{hom}({\odot}^kTX,\operatorname{End}(S))$ satysfiying the following properties
\begin{description}
\item[C1] The fibers $S_p$ are left modules over the Clifford algebras $\operatorname{Cl}(T_pX)$ i.e. there's a vector bundle homomorphism $TX\otimes S\longrightarrow S$, $Z\otimes \varphi \longmapsto Z \cdot \varphi$ such that $(Z\cdot Y+Y\cdot Z+2g(Z,Y))\cdot \varphi=0.$ 
\item[C2] The inner product is parallel, $d\langle \varphi,\psi \rangle= \langle \nabla^S \varphi,\psi \rangle + \langle \varphi,\nabla^S\psi \rangle.$
\item[C3] The Clifford action
$\operatorname{Cl}(X)\otimes S\longrightarrow S$ is parallel, $\nabla^S_Z(Y\cdot \varphi)=(\nabla^{g}_Z Y)\cdot \varphi +Y\cdot \nabla_Z^S \varphi.$
\item[C4]Clifford multiplication by tangent vectors is symmetric $\langle Z\cdot \varphi, \psi\rangle =\langle \varphi , Z\cdot \psi \rangle.$
\item[C5] The section $Q$ is parallel with respect to the connection on
$\operatorname{hom}({\odot}^kTX,\operatorname{End}(S))$ induced by the Levi--Civita connection $\nabla^g$ and $\nabla^S$.
\item[C6]$Q(\cdot,...,\cdot)$ is simmetric w.r.t. $\langle\cdot,\cdot\rangle$.
\item[C7] If $N$ is a future directed timelike vector then, putting 
$Q_N:=Q(N,...,N)$:
\begin{enumerate}
\item $Z\cdot Q_N=-Q_N Z \cdot, $ if $g(Z,N)=0$
\item $N \cdot Q_N=Q_N N\cdot.$
\end{enumerate}
\item[C8] If $N$ is a future directed timelike vector then the quadratic form $\langle\langle \cdot,\cdot \rangle \rangle_N$ defined by 

\noindent $\langle\langle \varphi,\psi \rangle \rangle_N=\langle Q_N\varphi,\psi\rangle $
is positive definite.  
\end{description}
From $C8$ it follows that $Q_N$ is invertible. From the first one of$C7$ it follows that its spectrum is symmetric w.r.t the origin, which together with with $C8$ implies that the bilinear form $\langle \cdot,\cdot \rangle$ has index $(1/2 \operatorname{rank}(S),1/2 \operatorname{rank}(S)).$
\end{dfn}
From $C7$ follows that \begin{description}
\item[C7']
\begin{enumerate}
\item For any vector $Z$ we have $Z \cdot Q_Y=-Q_YZ\cdot,$ $g(Y,Z)=0$
\item $Z \cdot Q_Z=Q_Z Z\cdot, $ again $Q_Z=Q(Z,...,Z)$. 
\end{enumerate}

\end{description}

\noindent With such a structure one can Define the generalized Dirac operator $D:C^{\ty}(X,S)\longrightarrow C^{\ty}(X,S)$ following the composition
$$\xymatrix{C^{\ty}(X,S)\ar[r]^-{i\nabla^S}&C^{\ty}(T^*X,S)\ar[r]^-{\sharp}&C^{\ty}(TX,S)\ar[r]^{\operatorname{Cl}}& C^{\ty}(X,S) }$$ 
where we used the musical isomorphism $\sharp$. In a local orthonormal frame
$e_0,...,e_n$ with $\varepsilon_j=g(e_j,e_j)$
 one can check the formula
$$D\varphi=i\sum_{j=0}^n\varepsilon_j \cdot \nabla_{e_j}^S\varphi.$$
It is a first order differential operator whose principal symbol is Clifford multiplication $$\sigma_D(\xi)=iZ\cdot  \textrm{ where }  \xi^\sharp=Z\textrm{ i.e.  }g(Z,\cdot)=\xi,$$ 
This can be seen immediately from the formula $D(f\varphi)=i \operatorname{grad}f \cdot \varphi+f D\varphi.$
In this sense one says that the Dirac operator is the quantization of the Clifford action.
\section{Propagation of singularities}
In this section $M$ is a manifold without boundary. We are going to explain the theorem about propagation of singularities for systems of real principal type of Denker \cite{denk} and apply to the generalized Dirac operator on  a Lorentzian manifold.
We shall use only classical pseudodifferential operators.
Now a scalar pseudodifferential operator on $X$ is said of real principal type if its principal symbol $\sigma$ is real and the Hamiltonian vector field $H_{\sigma}$ is non vanishing and does not have the radial direction on the zero set of $\sigma$. 
There is a corresponding notion for systems. An order $m$, $N\times N$ system $A$ of pseudodifferential operators with principal symbol $\sigma_{m}(A)$ is said of real principal type at $\eta\in T^*M$ if there exists 
on a neighborhood $U$ of $\eta$ an $N\times N$ symbol $\tilde{\sigma}$
 and a scalar symbol of real principal type $q$ such that 
\begin{equation}\label{pri}
\tilde{\sigma}(\xi)\circ \sigma_m(A)(\xi)=q(\xi)  \operatorname{Id},\,\xi \in U.
\end{equation}
 If this property holds for every $\eta$ in $T^*M$ we say $A$ is real principal type. Clearly the existence of $\tilde{\sigma}$ is only locally granted and $U$ can be assumed conical.

\begin{prop}
\label{didi}The generalized Dirac operator on a Lorentzian manifold is real principal type.
\end{prop}
\begin{proof}
Let $U\subset M$ be an open set in which a timelike future directed vector field $N$ exists. Remember that $Q_N$ is invertible. Define in $T^*U$ the following simbol, 
$$\widetilde{\sigma}(\xi)=-iQ_N^{-1}(aY-bN)Q_N$$
 if $\xi^{\sharp}=Z=aY+bN$ with $g(Y,N)=0.$ Then a straightforward computation based on $C7$ and $C1$ shows that
  $$\widetilde{\sigma}(\xi)\circ \sigma_D(\xi)=\|\xi\|_g^2 \operatorname{Id}$$ where
 $\xi\longmapsto\|\xi\|_g^2$ is the Lorentzian norm, in coordinates $\|\xi\|_g^2=g^{ij}(x)\xi_i \xi_j$ a well known Hamiltonian who generates the geodesic flow. It is standard the computation that it is of real principal type. The zero set is the light cone and the null bicharacteristics are lightlike geodesic. 
\end{proof}
\noindent For a real principal type system there's an elegant result of Denker about the propagation of singularities \cite{denk,kra}. First one defines the $C^{\ty}$--based polarization set of a vector valued  distribution.
It is a family of vector spaces in $\C^N$ (or a vector bundle $S$) over the singular support of the distribution conical in the $\xi$--variable and serves as an indicator of the direction of the strongest singularity. Here the definition:
Let $u\in \mathcal{D}'(M,\mathbb{C}^N)$ a vector valued distribution i.e. a vector of distributions $u_i\in \mathcal{D}'(M)$. It's wave front set is by definition the union of all the wave front sets of its components
$$WF(u):=\bigcup_{i=1,..N}WF(u_i).$$ In particular it does not contain any information about the components of distributions that are singular. In order to specify the singular directions in the vector space $\mathbb{C}^N$, one could consider vector--valued operators that map the vector valued distribution to a smooth scalar function, instead of just looking at scalar operators mapping the individual components to smooth functions. This approach leads to the definition of the \underline{polarization set}.
\begin{dfn}The polarization set of a distribution $u\in \mathcal{D}'(M,\mathbb{C}^N)$ is defined as 
$$WF_{\operatorname{pol}}(u):=\bigcap_{Au\in C^{\ty}(M)}\mathcal{N}_A.$$ Here, for an operator mapping the vector valued distribution to a smooth function, with (principal) symbol $a(x,\xi)$ define $$\mathcal{N}_A:=\{(x,\xi;w)\in (T^*M\setminus 0)\times \mathbb{C}^N:w\in \operatorname{ker}a(x,\xi)\}.$$ 
\end{dfn}
We stress that the intersection is taken over all $1\times N$ systems $A\in L^0(M)^N$ of classycal $\Psi$DOs. Two important properties proved by Dencker \cite{denk}

Let $u\in \mathcal{D}'(M,\mathbb{C}^N)$ and $\pi_{1,2}:T^*M\times \mathbb{C}^N\longrightarrow T^*M$ be the projection on the cotangent bundle, then
$$\pi_{1,2}(WF_{\operatorname{pol}}(u)\setminus (T^*M\times 0))=WF(u).$$ In this way the polarization set is a refinement of the wave front set.

Let $A$ be an $N\times N$ system of pseudo--differential operators on $M$ with principal symbol $a(x,\xi)$ and $u\in \mathcal{D}'(M,\mathbb{C}^N)$. Then
$$a(WF_{\operatorname{pol}}(u))\subset WF_{\operatorname{pol}}(Au),$$ where $a$ acts on the fibre: $a(x,\xi;w)=(x,\xi;a(x,\xi)w).$

 If $E$ is an $N\times N$ system of pseudo--differential operators on $M$ and its principal symbol $e(y,\eta)\neq0$, then
$$e(WF_{\operatorname{pol}}(u))=WF_{\operatorname{pol}}(Eu)$$ over a conic neighbourhood of $(y,\eta).$ 
 From this last property we see that the polarization set behaves covariantly under a change of coordinates. Thus the polarization set can be defined for distributional sections 
$\mathcal{D}'(M,E)$ 
of a vector bundle $E$ giving a well defined subset
$$WF_{\operatorname{pol}}(u)\subset \pi^*E$$ of the vector bundle lifted over the cotangent bundle.

\noindent So return to the setting before and assume we are given an $N\times N$ system $A$ of pseudodifferential operators acting on $\C^N$ of principal type and \eqref{pri} valid in $U$. Define the set $$\Omega_A:=\{\xi \in T^*X:\operatorname{det}\sigma_m(A)(\xi)=0
\}$$ locally $\Omega_A=q^{-1}(0).$ 
Furthermore one can show that $A$ is real principal type around $\eta$ if and only if around $\eta$ the set $\Omega_A$ is a hypersurface with non radial Hamiltonian field, the dimension of $\operatorname{ker}\sigma_m(A)$ is constant in $\Omega_A$ and for every normal vector $\rho \in N(\Omega_A)$ if $$\pi_C:\C^N\longrightarrow C^N/\operatorname{Im}\sigma_m(A)(\xi)=\operatorname{coker}\sigma_m(A)(\xi)$$ is the quotient mapping then  
the map
$$\pi_C(\partial_{\rho})\sigma_m(A)(\xi):\operatorname{ker}\sigma_m(A)(\xi)\longrightarrow \operatorname{coker}\sigma_m(A)(\xi)$$ is an isomorphism.

\noindent Now let $u\in C^{-\ty}(X,\C^N)$ a distribution such that $Au\in C^{\ty}$. Let denote ${\mathcal{N}_{A}}_{|_{\xi}}\subset \C^N$ the kernel of $\sigma_m(A)$ at $\xi$. Now if $(\xi,w)\in \mathcal{N}_A$, by definition it follows necessarily that the (vector valued) distribution wave front set is contained in $\Omega_A$.
$$\operatorname{WF}(u)\subset \Omega_A.$$
One can show as in the scalar Duistermaat H\"ormander result \cite{duis} the wave front set is union of null bicharacteristics of $q$ in $\Omega_A$.
Well here $q$ is not unique but it can be shown to be unique modulo the multiplication by a never vanishing function on $T^*M$. This does not affect the bicharacteristics since the wave front is conical in the $\xi$--variable. We shall call them bicharacteristics again. It remains to study the polarization vectors over these bicharacteristics. The idea of Denker is to introduce a partial connection and show that they follow the parallel transport along these null bicharacteristics and these vector fields form completely the polarization set.
The Denker connection depends also from the subprincipal symbol (\cite{duis} Sect 5.2)
of the operator\footnote{this is a genuine feature of the system nature, in the scalar case only the principal symbol contributes}
The principal symbol of $A$ has an asymptotic homogeneous expansion
$$\sigma(A)(\xi)=\sigma_m(A)(\xi)+p_{m-1}(\xi)+p_{m-2}(\xi)+,.....$$
The subprincipal symbol is defined as
$$\sigma(A)^s(\xi)=p_{m-1}(\xi)-\dfrac{1}{2i} \dfrac{\partial^2  \sigma_{m}(A)(\xi)}{\partial x^j \partial \xi_j},$$
form the "poisson bracket" of matrices (this is not anti symmetric !)
$$\{\tilde{\sigma},\sigma_m(A)\}:= \dfrac{\partial \tilde{\sigma}}{\partial_{\xi_j}}\dfrac{\partial \sigma_m(A)}{\partial x^j}-\dfrac{\partial \tilde{\sigma}}{\partial x^j}\dfrac{\partial \sigma_m(A)}{\partial \xi_j}.
$$ let $H_q$ denote the Hamiltonian vector field of $q$ defined by 
$dq(Y)=\omega(H_q,Y)$ for every $Y\in TM$, then for a smooth section $w$ of the trivial bundle $(T^*X\setminus o)\times \C^N$
define the Denker connection 
\begin{equation}\label{denko}
\nabla^Aw=H_q w +1/2\{\tilde{\sigma},\sigma_m(A)\}w+i\tilde{\sigma}\sigma(A)^sw.\end{equation}
 One can show that $\nabla^A$ resctricts to $\mathcal{N}_A$ i.e. if $w$ is a vector field along a bicharacteristic $\gamma$ in $\Omega_A$ then 
$$\nabla^A_{\dot{\gamma}}w \in \operatorname{ker}\sigma_m(A)(\gamma(t))={\mathcal{N}_A}_{|\gamma(t)}\textrm{ iff }w(\gamma(t))\in \operatorname{ker}\sigma_m(A).$$
This means that the equation $\nabla^Aw=0$ can be solved in $\mathcal{N}_A$ which is a necessary condition to be in $\operatorname{WF}_{\operatorname{pol}}(u).$
\begin{dfn}(Denker \cite{denk}) A Hamiltonian orbit of the system $A$ is a line bundle $L\subset {\mathcal{N}_A}_{|\gamma}$ where,
\begin{itemize}
\item $\gamma$ is an integral curve of the Hamiltonian field of of $\Omega_{A}$.
\item $L$ is spanned by a $C^{\ty}$ section $w$ such that $\nabla^Aw=0.$
\end{itemize}
\end{dfn}
Finally we can state the theorem about the propagation of the polarization set.
\begin{thm}(Denker \cite{denk})

\noindent Suppose $\eta\in \Omega_A$ is a point where $A$ is real principal type. If
$u$ is a vector valued distribution such that 
 $\eta\notin \operatorname{WF}(Au)$ (vector valued wave front)
 then over a neighborhood of $\eta$ in $\Omega_A$, $\operatorname{WF}_{\operatorname{pol}}$ is a union of Hamiltonian orbits of $A$.
 
 \noindent If $Au\in C^\ty$ we are saying that that the non trivial part of the polarization set of $u$ is parallel in $\mathcal{N}_A$ along the flow of $\Omega_A.$
 \end{thm}
 \noindent We are going to apply this result to the generalized Dirac operator. We know from proposition \ref{didi} that $\Omega_D=\{\xi :\|\xi\|^2_g\}$ is the light cone. Then the null bicharacteristics are the null geodesics. Let's compute explicitly the Denker connection.
  \begin{thm}The Denker connection for the generalized Dirac operator is $\nabla^S$ lifted to the cotangent bundle (along null geodesics), $\pi:T^*X\longrightarrow X.$
  
\noindent In particular the polarization set of the generalized Dirac operator is union of Hamiltonian orbits i.e. curves
$t\longmapsto \gamma(t)=(x(t),\xi(t);s(t))\in \pi^*S$ such that
\begin{itemize}
\item $x(t)$ is a null geodesic,
\item $\xi(t)^\sharp$ is the velocity of the geodesic
\item $s(t)$ is a section of $S$ along $\gamma$ parallel with respect to the connection $\pi^*\nabla^S$.
\end{itemize}
 \end{thm}
 \begin{proof}
 %Let $\gamma$ a null bicharacteristic such that $\dot{\gamma}(t_0)=$
  %$p$ such that $\pi_{*}(\dot{\gamma}(t_0))=Z$ choose coordinates around $p=\pi(\gamma(t_0)$ such that
  %$\nabla^g_{\partial_{x_i}}Z$=0. Lift $\nabla^S$ to $T^*X$. 
 \noindent On a point $(p,\xi)$ with $\xi^\sharp=Z=aY+bN$ on the light cone choose coordinates such that $\nabla^g_{{\partial_x}_k}Z=0$ and $\nabla^g_{{\partial_x}_k}(aY-bN)=0$. This is possible since $Z$ is lightlike meaning that $Y\neq0$.
 From $C8$ since we know the signature of 
 $\langle \cdot,\cdot \rangle$ on $S$ we can choose a trivialization $s_1,...,s_{N/2},s_{N/2+1},...,s_N$ such that $(\nabla^S_{e_i}e^j)(p)=0.$ Now the principal symbol $\sigma_{1}(D)$ over $T^*U$ corresponds to a matrix $$f_{ij}(x,\xi)=\langle (\sigma(D)(\xi)s_i,s_j)\varepsilon_{ij}=\langle -iZ\cdot s_i,s_j\rangle$$ where $\varepsilon_{ij}=\langle s_i,s_j\rangle$. Then in the point $(p,\xi)$ we have
 $$\partial_{x_k}f_{ij}=(\pi^*\nabla^S)_{\partial_{x_k}}(-iZ\cdot s_i,s_j)\varepsilon_{ij}=-i\langle \nabla^S_{\partial_{x_k}}(Z\cdot s_i),s_j\rangle \varepsilon_{ij}=-i\langle (\nabla^g_{\partial_{x_k}}Z)\cdot s_i,s_j\rangle=0$$
 stating no more than the compatibility with the Clifford action, the Levi Civita connection and the inner product on $S$.
  and a corresponding statement for 
 the symbol $\tilde{\sigma}$. In particular in the point $\xi$ on the lightcone the second and third terms in \eqref{denko} vanish and remains only the Lie derivative along 
 $H_{\|\cdot\|^2_g}$. The same for $\pi^*\nabla^S$.
 \end{proof}
 %\headheight=0pt

%\headsep=0pt

%\footskip=0pt

%\begin{center}
%\par\vspace {25mm}
%{\Large\sc    La Sapienza Universit\`a di \ Roma \ }\\[4ex]

%{\large Facolt\`a di Scienze Matematiche Fisiche e Naturali}\\[2ex]
%{\large Dottorato di Ricerca in Matematica }\\[1ex]
 %{\large XX Ciclo}\\

%\end{center}

%\begin{center}
%\par\vspace {35mm}

%{\Huge\bf A signature formula for foliations on manifolds with cylindrical ends }  \\

%\bigskip

%\bigskip

%{\Huge\bf  }

%\par\vspace{45mm}
%\end{center}

\end{document}